\documentclass[12pt]{amsart}
\usepackage{amssymb}
\usepackage{amsmath}
\usepackage[hidelinks]{hyperref}
\usepackage{etoolbox}
\usepackage{enumitem}
\usepackage{xcolor}
\usepackage{tikz-cd}
\usepackage[normalem]{ulem}
\hypersetup{
	colorlinks,
	linkcolor={red!30!black},
	citecolor={blue!50!black},
	urlcolor={blue!80!black}
}

\newtheorem{theorem}{Theorem}

\newtheorem{proposition}[theorem]{Proposition}

\newtheorem{definition}[theorem]{Definition}
\newtheorem{corollary}[theorem]{Corollary}

\theoremstyle{remark}
\newtheorem{remark}[theorem]{Remark}

\newcommand{\p}{\vskip .4cm}

\newcommand{\Z}{\mathbb{Z}}
\newcommand{\Q}{\mathbb{Q}}
\newcommand{\R}{\mathbb{R}}
\newcommand{\C}{\mathbb{C}}
\newcommand{\Cc}{C^{\times}}

\newcommand{\cB}{\mathcal{B}}

\newcommand{\cO}{\mathcal{O}}

\newcommand{\bG}{\mathbb{G}}

\newcommand{\til}{\tilde}
\newcommand{\bsl}{\backslash}
\newcommand{\ra}{\rightarrow}

\newcommand{\sra}{\twoheadrightarrow}

\newcommand{\Lie}{\operatorname{Lie}}
\newcommand{\rank}{\operatorname{rank}}
\newcommand{\depth}{\operatorname{depth}}
\newcommand{\Ad}{\operatorname{Ad}}

\newcommand{\Irr}{\operatorname{Irr}}

\newcommand{\fg}{\mathfrak{g}}

\newcommand{\fm}{\mathfrak{m}}

\newcommand{\slt}{\mathfrak{sl}_2}

\newcommand{\Hom}{\operatorname{Hom}}

\newcommand{\im}{\operatorname{im}}

\newcommand{\SL}{\mathrm{SL}}

\newcommand{\dMP}{\operatorname{dMP}}
\newcommand{\DM}{\operatorname{DM}}
\newcommand{\DI}{\operatorname{DI}}
\newcommand{\DS}{\operatorname{DS}}

\newcommand{\uDM}{\underline{\operatorname{DM}}}

\newcommand{\nil}{\fg^{\mathrm{nil}}/\!\!\sim}
\newcommand{\nild}{(\fg^*)^{\mathrm{nil}}/\!\!\sim}

\begin{document}

\title{Local character expansion for mod-$\ell$ representations}

\author{Cheng-Chiang Tsai}
\thanks{The author is supported by Taiwan NSTC grants 114-2115-M-001-009 and 114-2628-M-001-003.}
\address{Institute of Mathematics, Academia Sinica, 6F, Astronomy-Mathematics Building, No. 1,
Sec. 4, Roosevelt Road, Taipei, Taiwan \vskip.2cm
also Department of Applied Mathematics, National Sun Yat-Sen University, and Department of Mathematics, National Taiwan University}

\email{chchtsai@gate.sinica.edu.tw}

\begin{abstract} 
    Let $G$ be a $p$-adic reductive group with $p$ ``very large.'' For any irreducible admissible representation $\pi$ of $G$ over an algebraically closed field $C$ of characteristic $\not=p$, we define a ``local character expansion'' of $\pi$ with coefficients $c_{\mathcal{O}}(\pi)\in\Q$, that does not use the character of $\pi$ directly but instead use the multiplicities of degenerate Moy-Prasad types. The existence of such local character expansion for mod-$\ell$ representations is shown in another paper of the author using a different and quicker method.
\end{abstract}

\makeatletter
\let\@wraptoccontribs\wraptoccontribs
\makeatother

\maketitle

\tableofcontents

\section{Introduction}

Let $F$ be a non-archimedean local field and $k$ its residue field. Let $\bG$ be a connected reductive group over $F$, with the assumption that $p=\mathrm{char}(k)$ is very large compared to $\bG$; see \S\ref{sec:conv} for detail. Write $G=\bG(F)$, $\fg=(\Lie\bG)(F)$ and $\fg^*:=(\Lie^*\bG)(F)$. Consider an algebraically closed field $C$ of characteristic $\ell\not=p$, possibly $\ell=0$. Take an additive character $\til{\psi}:(F,+)\ra\Cc$ that is trivial on $\fm_F$ but non-trivial on $\cO_F$, so that it restricts to an additive character $\psi:(k,+)\ra\Cc$. 

Denote by $\nild$ the set of nilpotent $G$-orbits in $\fg^*$. When $\mathrm{char}(F)=0$ and $C=\C$, the famous Harish-Chandra--Howe local character expansion \cite[Theorem 16.2]{HC} asserts that for any finite-length admissible $\C$-representation $\pi$, its character $\Theta_{\pi}$ enjoys an expansion on some neighborhood $U_{\pi}$ of $\mathrm{id}\in G$: there exist constants $\til{c}_{\cO}(\pi,\til{\psi})\in\C$ for $\cO\in\nild$ such that for every $\C$-valued smooth compactly supported function $f$ supported on $U_{\pi}$ we have
\begin{equation}\label{eq:LCE}
\Theta_{\pi}(f)=\sum_{\cO\in(\fg^*)^{\mathrm{nil}}/\sim}\til{c}_{\cO}(\pi,\til{\psi})\cdot\mu_{\cO}(\widehat{f\circ\exp})
\end{equation}
where 
\[
\widehat{f\circ\exp}(X):=\int_{Y\in\fg}\til{\psi}(\langle X,Y\rangle) f(\exp(y)).
\]
and $\mu_{\cO}$ is integration on $\cO$, both to be normalized in \S\ref{sec:conv}.

Meanwhile, Moy and Prasad defined certain types using the Moy-Prasad filtration: For a point $x\in\cB(G)$ in the reduced Bruhat-Tits building, they defined open compact subgroups $G_{x\ge s}\subset G$ for $s\in\R_{\ge 0}$ as well as lattices $\fg_{x\ge s}\subset\fg$, $\fg^*_{x\ge s}\subset\fg^*$ for $s\in\R$. We also write $G_{x=s}$ etc. for the quotient in the filtration. See \S\ref{sec:conv} for a review of these notations as well as an explanation for the convention. For $s>0$, given any $\phi\in\fg_{x=-s}^*$ we can construct
\[
\def\arraystretch{1.2}
\begin{array}{ccccccccl}
\psi_{\phi}&:&G_{x\ge s}&\sra& G_{x=s}&\cong&\Hom_k(\fg^*_{x=-s},\,k)&\ra&\Cc\\
&&&&&&\lambda&\mapsto &\psi(\lambda(\phi))
\end{array}
\]
Here the middle isomorphism follows from \cite[Theorem 13.5.1(1)]{KP23}. We say $\phi$ and $\psi_{\phi}$ are {\bf degenerate} if $\phi+\fg^*_{x>-s}$ contains a nilpotent element in $\fg^*$. We say $\pi$ contains $\psi_{\phi}$ whenever $\Hom_{G_{x\ge s}}(\psi_\phi,\pi)\not=0$. It follows from \cite[Theorem 5.2 and \S7.1]{MP94}\footnote{It is stated only for $\bG$ absolutely quasi-simple simply connected, but the proof works in general as long as $\bG$ is tamely ramified, which follows from our assumption on $p$.} when $\mathrm{char}(C)=0$ and from \cite[Th\'{e}or\'{e}me II.5.7.a]{Vig96} when $\mathrm{char}(C)\not=p$ that for any finite-length admissible $C$-representation $\pi$ there is a quantity $\depth(\pi)\in\Q_{\ge 0}$ such that when $s>\depth(\pi)$, any $\psi_{\phi}$ constructed above contained by $\pi$ is degenerate. We consider 
\[\dMP(s):=\{(x,\phi)\;|\;x\in\cB(G),\;\phi\in\fg^*_{x=-s}\text{ is degenerate}\},\]
For $\Omega\subset\fg^*$ a compact open subset we denote by $[\Omega]$ the function that takes value $1$ on $\Omega$ and $0$ elsewhere. Our main theorem is

\begin{theorem}\label{thm:main} Let $F$, $\bG$ and $G$ be as before. Then for any $C$, $\psi$ as above and any finite-length admissible $C$-representation $\pi$, there exist constants $c_{\cO}(\pi,\psi)\in\Q$ for $\cO\in\nild$ satisfying
\begin{equation}\label{eq:LCE2}
\dim_C\Hom_{G_{x\ge s}}(\psi_\phi,\pi)=\sum_{\cO\in(\fg^*)^{\mathrm{nil}}/\sim}\mu_{\cO}([\phi+\fg^*_{x>-s}])\cdot c_{\cO}(\pi,\psi)
\end{equation}
for any $s>\depth(\pi)$ and $(x,\phi)\in\dMP(s)$.

Equation \eqref{eq:LCE2} uniquely determines $c_{\cO}(\pi,\psi)$. Moreover, when $\mathrm{char}(F)=0$ and $C=\C$, we have $c_{\cO}(\pi,\psi)=\til{c}_{\cO}(\pi,\til{\psi})$ where the latter is as in \eqref{eq:LCE}, normalized as in \cite[I.8]{MW87}.
\end{theorem}

As mentioned in the abstract, local character expansions with $\Q$-coefficients for mod-$\ell$ representations are also shown to exist in \cite[Corollary 5 and 6]{Tsa25a}. We present \cite{Tsa25a} as the ``main paper'' because we consider the method {\it op. cit.} to (have the potential to) give better results. The assumption on $p$ is however different; see Remark \ref{rmk:poschar}.

\begin{remark} The fact that $\til{c}_{\cO}(\pi,\til{\psi})\in\Q$ when $C=\C$ was shown in \cite[\S7]{Var14}.
\end{remark}

\subsection*{Acknowledgment} I am truly grateful to Marie-France Vign\'{e}ras for a beautiful talk which inspired this work, and for many very helpful discussions after that. 
I sincerely thank Rahul Dalal and Mathilde Gerbelli-Gauthier for their generous and stimulating explanation about their related upcoming paper \cite{DGGM}\footnote{I think there is no overlap of content, but it doesn't hurt to emphasize that the cited work was essentially done and explained to me before this work of mine started.}, and to Professor James Arthur and well as Clifton Cunningham, Melissa Emory, Paul Mezo and Bin Xu for the wonderful conference in the Fields Institute in August 2025. Lastly, I owe an immense debt of gratitude to Stephen DeBacker for teaching me about many things related to his works and for numerous other enlightening discussions that very much influence this work. 

\section{Hypotheses and conventions}\label{sec:conv}

Recall that $F$ is our non-archimedean local field, $k$ its residue field and $\bG$ be a connected reductive group over $F$ so that $G=\bG(F)$. Our work is built very heavily and essentially on the work of DeBacker \cite{De02b}, and has many elements in common with the algorithms in \cite{De02a} and \cite{Wa06}.

Throughout this paper, unless otherwise stated we assume the hypotheses in \cite[\S4.2]{De02b} hold, which can be guaranteed if $p=\mathrm{char}(k)$ satisfies $p\ge\max(\rank_{\bar{F}}\bG,\,271)$; see \cite[Appendice 3]{Wa06} which verifies some even stronger hypotheses. We discuss later in Remark \ref{rmk:p} about optimal condition on $p$.

We denote by $G_{x\ge r}$ ($r\ge 0$) the Moy-Prasad subgroup that is denoted by $P_{x,r}$ in \cite[\S2.6]{MP94} and by $G_{x,r}$ in most recent works. This notation is extensively used in \cite{Spi18} and \cite{Spi21}. It has the following advantage for us: with respect to any apartment that contains $x$, the affine root subgroups contained in $G_{x\ge r}$ are exactly those affine roots that take value $\ge r$ at $x$. The same holds for $G_{x>r}:=\bigcup_{s>r}G_{x,s}$. Moreover, for the Moy-Prasad (sub)quotient $G_{x=r}:=G_{x\ge r}/G_{x>r}$, the affine root subquotients that appears in $G_{x=r}$ exactly come from those affine roots that take value $=r$ at $x$. Same for $\fg_{x\ge r}$, $\fg_{x=r}$, $\fg^*_{x=r}$, etc..




\section{Some elementary linear algebra}

Let $F$, $\bG$, $G$, $\fg$ and $\fg^*$ be as in the introduction. 
For any $r\ge 0$, $C$ as before, non-trivial additive $\psi:k\ra\Cc$ and any finite-length admissible $C$-representation $\pi$ with $\depth(\pi)\le r$, we are concerned with the following data of multiplicities
\[
v_{\pi,\psi}^{>r}=\left(\dim_C\Hom_{G_{x\ge s}}(\psi_{\phi},\pi)\right)_{s>r,(x,\phi)\in\dMP(s)}\in\uDM^{>r}:=\prod_{s>r,\,(x,\phi)\in\dMP(s)}\Q.
\]

Denote by
\[\DM^{>r}\subset\uDM^{>r}\]
the $\Q$-subspace of the infinite-dimensional space spanned by 
$v_{\pi,\psi}^{>r}$ for all possible $C$, $\psi$ and $\pi$ with $\depth(\pi)\le r$. Also consider some data of integrals:
\begin{equation}
    v_{\cO}^{>r}:=\left(\mu_{\cO}([\phi+\fg^*_{x>-s}])\right)_{{s>r},(x,\phi)\in\dMP(s)}\in\uDM^{>r}
\end{equation}
\[
\DI^{>r}:=\operatorname{span}_{\Q}\langle v_{\cO}^{>r}\;|\;\cO\in\nild\rangle\subset\uDM^{>r}.
\]
Equip $\nild$ with the usual partial order that $\cO_1\le \cO_2\Leftrightarrow\cO_1\subset\overline\cO_2$ where the latter is the closure in the $p$-adic topology. For degenerate $\phi\in\fg^*_{x=-s}$ we have the DeBacker lift \cite[Cor. 5.2.5, Lemma 5.3.3]{De02b} $\cO_{x,-s,\phi}\in\nild$ which is the unique smallest orbit in $\nild$ that intersects the coset $\phi+\fg^*_{x>-s}$. Since $\phi+\fg^*_{x>-s}$ is open, it intersects a nilpotent orbit $\cO\in\nild$ iff $\cO\ge\cO_{x,-s,\phi}$. In other words,
\begin{equation}\label{eq:triang}
\begin{array}{llll}
\mu_{\cO}([\phi+\fg^*_{x>-s}])&\not=&0,&\text{if }\cO\ge\cO_{x,-s,\phi},\\
\mu_{\cO}([\phi+\fg^*_{x>-s}])&=&0,&\text{if }\cO\not\ge\cO_{x,-s,\phi}.\\
\end{array}
\end{equation}
Moreover, for any $s$ and any $\cO\in\nild$ there exists $(x,\phi)$ such that $\cO_{x,-s,\phi}=\cO$. From this one easily sees that $\dim_{\R}\DI^{>r}=\#\nild$. 

For $v\in\uDM^{>r}$, $s>r$ and $(x,\phi)\in\dMP(s)$ we denote by $v_{s,(x,\phi)}\in\Q$ the corresponding component of $v$. 
Our main theorem will rely on the following property of the space $\DM^{>r}$.

\begin{proposition}\label{prop:main} Let $r\ge 0$ and $s_0,s_1>r$. Let $(x_i,\phi_i)\in\dMP(s_i)$ for $i=0,1$ be such that $\cO_{x_0,-s_0,\phi_0}=\cO_{x_1,-s_1,\phi_1}$. Then there exists a rational number $c\in p^{\Z}:=\{p^n\;|\;n\in\Z\}$ and a finite collection $c_j'\in\Z[1/p]$, $s_j'\ge\min(s_0,s_1)$ and $(x_j',\phi_j')\in\dMP(s_j')$ with $\cO_{x_j',s_j',\phi_j'}\gneq\cO_{x_0,-s_0,\phi_0}$, such that for every $v\in\DM^{>r}$ as well as $v\in\DI^{>r}$ we have 
\begin{equation}\label{eq:equation}
v_{s_1,(x_1,\phi_1)}=c\cdot v_{s_0,(x_0,\phi_0)}+\sum_j c_j'\cdot v_{s_j',(x_j',\phi_j')}
\end{equation}
\end{proposition}

The proof of Proposition \ref{prop:main} will be the topic of \S\ref{sec:main}. To prove the main theorem we consider

\begin{definition} A subset $W\subset\nild$ is {\bf closed} if $\bigsqcup_{\cO\in W}\cO\subset(\fg^*)^{\mathrm{nil}}$ is closed, i.e. $\cO_1\in W\text{ and }\cO_1\ge\cO_2\implies\cO_2\in W$.
\end{definition}
For any closed $W\subset\nild$, let 
\[
\uDM^{>r}_W:=\{v\in\uDM^{>r}\;|\;v_{s,(x,\phi)}=0\text{ if }\cO_{x,-s,\phi}\not\in W\}.
\]
and
\[
\DM_{W}^{>r}:=\DM^{>r}\cap\uDM_{W}^{>r},\;\;\DI_{W}^{>r}:=\DI^{>r}\cap\uDM_{W}^{>r}.
\]
We have
\begin{corollary}\label{cor:dim} Assume Proposition \ref{prop:main}. Let $r\ge 0$ and $W\subset\nild$ be closed. Then the space $\DI_W^{>r}$ is equal to the collection of vectors in $\uDM_W^{>r}$ that satisfy \eqref{eq:equation} for all $s_1$, $s_2$, $(x_1,\phi_1)$ and $(x_2,\phi_2)$ in Proposition \ref{prop:main} for any collection of choices of $c$, $c_i'$, $s_i'$, and $(x_i',\phi')$. Moreover
$\langle v_{\cO}^{>r}\;|\;\cO\in W\rangle$ is a basis for $\DI_W^{>r}$.
\end{corollary}

\begin{proof} Denote by $\DS_W^{>r}\subset\uDM_W^{>r}$ the subspace satisfying \eqref{eq:equation} for any collection of choices. Proposition \ref{prop:main} asserts that $\DI_W^{>r}\subseteq\DS_W^{>r}$ and we have to prove $\DS_W^{>r}=\DI_W^{>r}$. When $W$ is empty the assertion is vacuous. Suppose $W$ is non-empty. Let $\cO\in W$ be maximal, i.e. such that $W':=W\bsl\{\cO\}$ is closed. Then $v_{\cO}^{>r}\in\DI_W^{>r}\bsl\DI_{W'}^{>r}$ by the property of DeBacker lift, and $\DI_W^{>r}$ is spanned by $\DI_{W'}^{>r}$ and $v_{\cO}^{>r}$ which proves the last statement by induction.

Choose an $s'>r$ and $(x',\phi')$ with $\cO_{x',-s',\phi'}=\cO$. Using \eqref{eq:equation} in which $\cO_{x_i',s_i',\phi_i'}\gneq\cO_{x_1,s_i,\phi_i}$, we observe that for any $v\in\DS_W^{>r}$ with $v_{s',(x',\phi')}=0$, we have $\cO_{x,-s,\phi}=\cO\implies v_{s,(x,\phi)}=0$, i.e. $v\in\DS_{W'}^{>r}$.
Since $(v_{\cO})_{s',(x',\phi')}\not=0$, any $v\in\DS_W^{>r}$ can therefore be written as a linear combination of $v_{\cO}$ and a vector in $\DS_{W'}^{>r}$. By induction we may suppose $\DS_{W'}^{>r}=\DI_{W'}^{>r}$ and thus $v\in\DI_W^{>r}$ as required.
\end{proof}

\begin{proof}[Proof of Theorem \ref{thm:main}] By Corollary \ref{cor:dim} for $W=\nild$, we have that $\DM^{>r}\subset\DI^{>r}=\operatorname{span}_{\Q}\langle v_{\cO}^{>r}\rangle$ for any $r$. For any $C$, $\psi$ and $\pi$ as before, let $r=\depth(\pi)$ with which we can write 
\[
v_{\pi,\psi}^{>r}=\sum_{\cO\in(\fg^*)^{\mathrm{nil}}/\sim}c_{\cO}(\pi,\psi)\cdot v_{\cO}^{>r}.
\]
Then Theorem \ref{thm:main} holds with these $c_{\cO}(\pi,\psi)$. These coefficients are unique because $v_{\cO}^{>r}$ is a basis.

Now suppose $C=\C$. Then for sufficiently large $s$ we have
\[
\dim_C\Hom_{G_{x\ge s}}(\psi_\phi,\pi)=\Theta_{\pi}(\hat{f}\circ\log),\;\;f=[\phi+\fg^*_{x>-s}]
\]
plugging in \eqref{eq:LCE} then proves that $c_{\cO}(\pi,\psi)=\til{c}_{\cO}(\pi,\til{\psi})$.
\end{proof}

\begin{corollary}\label{cor:formula} Fix $r\ge 0$. For each $\cO\in\nild$, choose $s_{\cO}>r$, $(x_{\cO},\phi_{\cO})\in\dMP(s_{\cO})$ such that $\cO=\cO_{x_{\cO},-s_{\cO},\phi_{\cO}}$. Then there exists coefficients $A_{{\cO},{\cO}'}\in\Q$ for $\cO,\cO'\in\nild$ such that
\begin{enumerate}
    \item $A_{\cO',\cO}=0$ for any $\cO'\not\ge \cO$.
    \item For any $C$, $\psi$ and finite-length admissible $C$-representation $\pi$ with $\depth(\pi)\le r$ as in the introduction, we have
    \[    c_{\cO}(\pi,\psi)=\sum_{\cO'\ge\cO}A_{\cO',\cO}\cdot\dim_C\Hom_{G_{x_{\cO'}\ge s_{\cO'}}}(\psi_{\phi_{\cO'}},\pi).\]
\end{enumerate}
\end{corollary}

\begin{proof} We plug in $(s_{\cO'},(x_{\cO'},\phi_{\cO'}))$ to \eqref{eq:LCE2} and get a sytem of linear equations relating $c_{\cO}(\pi,\psi)$ and $\dim_C\Hom_{G_{x_{\cO'}\ge s_{\cO'}}}(\psi_{\phi_{\cO'}},\pi)$, with the matrix of the system given by $\mu_{\cO}([\phi_{\cO'}+\fg^*_{x_{\cO'}>-s_{\cO'}}])$. 
We then have the matrix $\left(A_{\cO',\cO}\right)$ equal to the inverse of the matrix $\left(\mu_{\cO}([\phi_{\cO'}+\fg^*_{x_{\cO'}>-s_{\cO'}}])\right)$, which are invertible triangular matrices thanks to \eqref{eq:triang}.
\end{proof}

\section{Proof of Proposition \ref{prop:main}}\label{sec:main}

Fix $r\ge 0$. In this section we fix a $G$-invariant symmetric non-degenerate bilinear form $\fg\times\fg\ra F$ such that it induces $\fg^*_{x=-s}\cong\fg_{x=-s}$ for every $x\in\cB(G)$ and $s\in\R$; such bilinear form exists \cite[Proposition 4.1]{AR00} under our assumption on $p$. Using the bilinear form, we shall identify $\fg^*$ with $\fg$ from now to use the theory $\slt$-triples. Let
\[\dMP:=\bigcup_{s>r}\{s,(x,\phi)\;|\;(x,\phi)\in\dMP(s)\}.\] Let us say $(s_0,(x_0,\phi_0)),\,(s_1,(x_1,\phi_1))$ are {\bf connected} if they satisfy the conditions in Proposition \ref{prop:main} for some $c\in p^{\Z}$, $c_i'\in\Z[1/p]$ and $(s_i',(x_i',\phi_i'))\in\dMP$. It is easy to see that
\begin{enumerate}
    \item The relation of being connected is an equivalence relation.
    \item For any $g\in G$ and $(s,(x,\phi))\in\dMP$, we have that $(s,(x,\phi))$ and $(s,(g.x,g.\phi))$ are connected.
\end{enumerate}
Suppose we are given $(s_0,(x_0,\phi_0)),\,(s_1,(x_1,\phi_1))\in\dMP$ such that $\cO_{x_0,-s_0,\phi_0}=\cO_{x_1,-s_1,\phi_1}$. We have to prove they are connected in the above sense. From \cite[Lemma 5.3.3]{De02b}, we have $\slt$-triple $(\Phi,H,E)\in\fg^3$ such that
\begin{enumerate}[label=(\alph*)]
    \item\label{a} $\Phi\in\phi_0+\fg_{x_0>-s_0}$, $H\in\fg_{x_0\ge 0}$ and $E\in\fg_{x_0\ge s_0}$.
    \item\label{b} \cite[Corollary 5.2.5]{De02b} asserts as a consequence of \ref{a} that $\Ad(G)\Phi=\cO_{x_0,-s_0,\phi_0}$, i.e. any nilpotent orbit $\cO$ that intersects $\phi+\fg_{x_0>-s_0}$ satisfies $\cO\ge\Ad(G)\Phi$.
\end{enumerate}
We similarly have another $(\Phi_1,H_1,E_1)$ satisfying \ref{a} and thus \ref{b} for $(s_1,(x_1,\phi_1))$. Since $\cO_{x_0,-s_0,\phi_0}=\cO_{x_1,-s_1,\phi_1}$ we have that $(\Phi_1,H_1,E_1)$ is conjugate to $(\Phi,H,E)$. 
Applying a $G$-conjugation on $(s_1,(x_1,\phi_1))$ as well as $(\Phi_1,H_1,E_1)$. We may suppose $(\Phi_1,H_1,E_1)=(\Phi,H,E)$, i.e. \ref{a} and thus \ref{b} above holds with $(s_1,(x_1,\phi_1))$ in place of $(s_0,(x_0,\phi_0))$ and the same $(\Phi,H,E)$.

Consider the geodesic $(t\mapsto x_t)_{t\in[0,1]}$ in the building $\cB(G)$ and also $s_t:=(1-t)s_0+ts_1\in\R$. Observe that
\begin{equation}\label{eq:geo}
\left(\fg_{x_0\ge s_0}\cap\fg_{x_1\ge s_1}\right)\subset\fg_{x_t\ge s_t}.
\end{equation}
Indeed, any affine root that takes value $\ge s_0$ at $x_0$ and value $\ge s_1$ at $x_1$ must take value $\ge s_t$ at $x_t$ for any $t\in[0,1]$. Equation \eqref{eq:geo} and similar equations imply
\[
\Phi\in\fg_{x_t\ge-s_t},\;H\in\fg_{x_t\ge 0}\text{, and }E\in\fg_{x_t\ge s_t}.
\]
Let $\phi_t:=\Phi+\fg_{x_t>-s_t}$, i.e. the image of $\Phi$ in $\fg_{x_t=-s_t}$. Then \ref{a} and therefore \ref{b} holds for $(s_t,(x_t,\phi_t))$. 
Moreover, because the Moy-Prasad lattice $\fg_{x\ge s}$ (resp. $\fg_{x>s}$) is upper semi-continuous (resp. lower semi-continuous) in both $x$ and $s$, there exist $0=t_0<t_1<t_2<...<t_m=1$ such that
\begin{enumerate}[label=(\roman*)]
    \item For any $i\in\{0,1,...,m-1\}$ and $u,v\in(t_i,t_{i+1})$ we have
\[
\begin{array}{l}
G_{x_u\ge s_u}=G_{x_v\ge s_v}\\
G_{x_u> s_u}=G_{x_v> s_v}\\
\fg_{x_u\ge s_u}=\fg_{x_v\ge s_v}\\
\fg_{x_u> s_u}=\fg_{x_v> s_v}\\
\fg_{x_u\ge-s_u}=\fg_{x_v\ge-s_v}\\
\fg_{x_u>-s_u}=\fg_{x_v>-s_v}
\end{array}
\]
    In particular $(s_u,(x_u,\phi_u))$ and $(s_v,(x_v,\phi_v))$ are connected for trivial reason.
    \item\label{ii} Consider any $i\in\{0,1,...,m\}$ and arbitrary $u\in(t_{i-1},t_{i+1})$ (for which we take $t_{-1}=t_0=0$ and $t_{m+1}=t_m=1$). 
    Write $s=s_{t_i}$, $x=x_{t_i}$, $\phi=\phi_{t_i}$, $\tau=s_u$, $y=x_u$ and $\varphi=\phi_u$. Then
\begin{equation}\begin{array}{l}
    G_{x>0}\subset G_{y>0}\subset G_{y\ge 0}\subset G_{x\ge 0}\\
    G_{x>s}\subset G_{y>\tau}\subset G_{y\ge\tau}\subset G_{x\ge s}\\
    \fg_{x>s}\subset \fg_{y>\tau}\subset \fg_{y\ge\tau}\subset \fg_{x\ge s}\\
    \fg_{x>-s}\subset \fg_{y>-\tau}\subset \fg_{y\ge-\tau}\subset \fg_{x\ge-s}\\
\end{array}\end{equation}
\end{enumerate}

It remains to show that $(s,(x,\phi))$ and $(\tau,(y,\varphi))$ as in \ref{ii} are connected, i.e. to prove Proposition \ref{prop:main} in this case. Note that $\varphi$ can be viewed as a coset $\varphi+\fg_{y>-\tau}$ which is a union of $\fg_{x>-s}$ cosets:
\begin{equation}\label{eq:fork}
\varphi+\fg_{y>-\tau}=\bigsqcup_{\chi\in\varphi+\fg_{y>-\tau}/\fg_{x>-s}}\chi+\fg_{x>-s}.
\end{equation}
The definition of DeBacker lifts implies that $\cO_{x,-s,\chi}\ge\cO_{y,-\tau,\varphi}$ for any degenerate $\chi$ in \eqref{eq:fork}. Suppose $\chi$ is such that $\cO_{x,-s,\chi}=\cO_{y,-\tau,\varphi}$. In this case \cite[Corollary 5.2.3]{De02b} asserts that
\begin{equation}\label{eq:fork2}
\begin{array}{ll}
&\left(\varphi+\fg_{y>-\tau}\right)\cap\Ad(G)\Phi\subset\Ad(G_{y>0})\Phi\\
\implies&\left(\chi+\fg_{x>-s}\right)\cap\Ad(G)\Phi\subset\Ad(G_{y>0})\Phi\subset\Ad(G_{x\ge 0})\Phi\subset\Ad(G_{x=0})\phi+\fg_{x>-s}\\
\implies&\chi\in\Ad(G_{x=0})\phi.
\end{array}
\end{equation}
Here $\cO_{x,-s,\chi}=\cO_{y,-\tau,\varphi}$ is used to guarantee that $\left(\chi+\fg_{x>-s}\right)\cap\Ad(G)\Phi$ is non-empty so that the last implication in \eqref{eq:fork2} holds. We have shown that in \eqref{eq:fork} we have either
\begin{enumerate}[label=(\Alph*)]
    \item $\chi$ is non-degenerate,
    \item\label{B} $\chi\in\Ad(G_{x=0})\phi$, or
    \item\label{C} $\cO_{x,-s,\chi}>\cO_{y,-\tau,\varphi}$.
\end{enumerate} Let $N>0$ be the number of $\chi$ in \ref{B}; this is a power of $p$ because such $\chi$ lie in a single orbit under $\im(G_{y>0}\ra G_{x=0})$, which is a finite $p$-group since $G_{y>0}$ is a pro-$p$ group. Let $\chi_1,...,\chi_m\in\fg_{x=-s}$ be a list of $\chi$ in case \ref{C}. For any $\cO\in\nil$ we have
\[
\mu_{\cO}([\varphi+\fg_{y>-\tau}])=N\cdot\mu_{\cO}([\phi+\fg_{x>-s}])+\sum_{j=1}^m\mu_{\cO}([\chi_j+\fg_{x>-s}]).
\]
This proves the part of Proposition \ref{prop:main} for $v\in\DI^{>r}$. Next, take any $C$ and $\psi$ as in the introduction and a finite-length admissible $C$-representation $\pi$ with $\depth(\pi)\le r$.
We have to prove
\[
\dim_C\Hom_{G_{y\ge\tau}}(\psi_{\varphi},\pi)=N\cdot\dim_C\Hom_{G_{x\ge s}}(\psi_{\phi},\pi)+\sum_{j=1}^m\dim_C\Hom_{G_{x\ge s}}(\psi_{\chi_j},\pi).
\]
Note that $\varphi$ is a $k$-linear functional on $\fg_{y=\tau}=\fg_{y\ge\tau}/\fg_{y>\tau}$ which can be pulled back to $\fg_{y\ge\tau}/\fg_{x>s}\subset\fg_{x=s}$. The set of its possible extensions to $\fg_{x=s}$ is exactly the indexing set in the RHS of \eqref{eq:fork}. Thus we have
\begin{equation}\label{eq:fork3}
\Hom_{G_{y\ge\tau}}(\psi_{\varphi},\pi)=\Hom_{G_{y\ge\tau}/G_{x>s}}(\psi_{\varphi},\pi^{G_{x>s}})
=\bigoplus_{\chi\in\varphi+\fg_{y>-\tau}/\fg_{x>-s}}\Hom_{G_{x=s}}(\psi_{\chi},\pi^{G_{x>s}}).
\end{equation}
\[
=\bigoplus_{\begin{subarray}{c}\chi\in\varphi+\fg_{y>-\tau}/\fg_{x>-s}\\\chi\text{ is degenerate}\\\cO_{x,-s,\chi}=\cO_{y,-\tau,\varphi}
\end{subarray}}\Hom_{G_{x=s}}(\psi_{\chi},\pi^{G_{x>s}})+\bigoplus_{\begin{subarray}{c}\chi\in\varphi+\fg_{y>-\tau}/\fg_{x>-s}\\\chi\text{ is degenerate}\\\cO_{x,-s,\chi}>\cO_{y,-\tau,\varphi}
\end{subarray}}\Hom_{G_{x=s}}(\psi_{\chi},\pi^{G_{x>s}}).\]
\[
=N\cdot\dim_C\Hom_{G_{x\ge s}}(\psi_{\phi},\pi)+\sum_{j=1}^m\dim_C\Hom_{G_{x\ge s}}(\psi_{\chi_j},\pi)
\]
as promised. We have thus proved Proposition \ref{prop:main} and therefore completed the proof of Theorem \ref{thm:main}.

\section{Additional remarks}

Let us revisit the proof in \S\ref{sec:main}. Instead of $s_0,s_1>r$ we suppose that $s_1>s_0=r$, and that $\pi$ is a finite-length admissible $C$-representation with $\depth(\pi)\le r$, possibly with equality. The proof becomes different only when $s=s_0=r<\tau$ in \ref{ii}, so that in \eqref{eq:fork} we no longer have $\chi$ degenerate; instead they could be the unrefined minimal $K$-types of Moy and Prasad in \cite[Definition 5.1]{MP94} and \cite[\S II.5.1]{Vig96}. (When $r=0$ this requires corresponding change of notations but otherwise works.)

Nevertheless, equation \eqref{eq:fork3} is still valid and implies that $\dim_C\Hom_{G_{y\ge\tau}}(\psi_{\varphi},\pi)$ can be computed in terms of some $\dim_C\Hom_{G_{x=r}}(\rho,\pi)$ for some $\rho\in\Irr_C(G_{x=r})$. More precisely we have

\begin{corollary} Let $F$, $\bG$, $G$, $C$ and $\psi$ be as in Theorem \ref{thm:main}. Fix $r\ge 0$. Then for any $\cO\in\nild$, there exists a finite number of $x_i\in\cB(G)$, $\rho_i\in\Irr_C(G_{x_i=r})$ and $c_i\in\Q$ such that for any finite-length $C$-representation $\pi$ with $\depth(\pi)\le r$ we have
\[
c_{\cO}(\pi,\psi)=\sum_ic_i\cdot\dim_C\Hom_{G_{x_i=r}}(\rho_i,\pi^{G_{x_i>r}}).
\]
\end{corollary}




Let us end with two remarks on our restriction on $p$.

\begin{remark}\label{rmk:p} The main theorems in \cite{De02b} implies that for any scalar $c\in 1+\fm_F$ and $\cO=\cO_{x,r,\phi}\in\nil$, we have $c\cdot \cO=\cO_{x,r,\bar{c}\phi}=\cO_{x,r,\phi}=\cO$ where $\bar{c}=1$ is the image of $c$ in $k$. For $G=\SL_p(\Q_p)$ we could have $c\cdot \dot\cO\not=\cO$, and consequently the main results in \cite{De02b} cannot hold with the same statements. It can also be verified that the data we consider in Theorem \ref{thm:main} is not enough to determine the local character expansion for $G=\SL_p(\Q_p)$ already when $C=\C$.

Conversely, we are not aware of any example - neither could we prove that there is no such example - in which $c\in 1+\fm_F\implies c\cdot \cO=\cO$, but either \cite[Theorem. 5.6.1]{De02b} does not hold or that our Theorem \ref{thm:main} does not hold.
\end{remark}

\begin{remark}\label{rmk:poschar} When $\mathrm{char}(F)>0$, the method in \cite{Tsa25a} for local character expansion has to assume \cite[Hypotheses 3.2.1]{De02a} about mock exponential map, while in this paper we have proved Theorem \ref{thm:main} only assuming $p$ to be very large. Consequently \cite[Corollary 10]{Tsa25a} holds with our condition in this paper. In general we expect \cite[Hypotheses 3.2.1]{De02a} to always hold under our assumption on $p$. When $\mathrm{char}(F)=0$ \cite[Corollary 5 and 10]{Tsa25a} needs no assumption on $p$.
\end{remark}

\p
\bibliographystyle{amsalpha}
\def\cfgrv#1{\ifmmode\setbox7\hbox{$\accent"5E#1$}\else \setbox7\hbox{\accent"5E#1}\penalty 10000\relax\fi\raise 1\ht7 \hbox{\lower1.05ex\hbox to 1\wd7{\hss\accent"12\hss}}\penalty 10000 \hskip-1\wd7\penalty 10000\box7}
\providecommand{\bysame}{\leavevmode\hbox to3em{\hrulefill}\thinspace}
\providecommand{\MR}{\relax\ifhmode\unskip\space\fi MR }
\providecommand{\MRhref}[2]{%
  \href{http://www.ams.org/mathscinet-getitem?mr=#1}{#2}
}
\providecommand{\href}[2]{#2}

\end{document}